\newif\ifarxiv
\arxivtrue

\RequirePackage{fix-cm}
\documentclass[smallextended]{svjour3}       
\smartqed  
\usepackage{graphicx}
\ifarxiv 
\journalname{ArXiv}
\fi
\usepackage{graphicx} 
\usepackage{amsfonts} 
\usepackage{amsmath}
\usepackage{amssymb}
\usepackage{algorithm}
\usepackage{algorithmic}

\usepackage{hyperref} 

\newtheorem{assumption}{Assumption}

\usepackage[textwidth=2cm, textsize=scriptsize]{todonotes}

\newcommand{\argmin}{\mathop{\mathrm{arg\,min}}}

\def\R{\mathbb{R}}

\newcommand{\rp}{\mathbb{R}^p}
\newcommand{\rn}{\mathbb{R}^n}

\newcommand{\est}{\hat\beta}

\newcommand{\lest}{\tilde\lambda}

\newcommand{\bt}{\beta^*}

\newcommand{\betatrue}{\beta^*}

\newcommand{\lasso}{\tilde \beta}

\newcommand{\consttrex}{c}
\newcommand{\consttrextwo}{c}

\renewcommand{\lasso}{\textcolor{red}{\tilde\beta}}
\renewcommand{\lasso}{{\tilde\beta}}
\newcommand{\lassoa}{\lasso(\lambda)}
\newcommand{\lassoat}{\lasso(\tilde\lambda)}
\newcommand{\lassosqrt}{\textcolor{red}{\bar\beta}}
\renewcommand{\lassosqrt}{{\bar\beta}}
\newcommand{\lassosqrta}{\lassosqrt(\gamma)}
\newcommand{\trex}{\textcolor{red}{\hat\beta}}
\renewcommand{\trex}{{\hat\beta}}

\newcommand{\scaledlassotext}{scaled LASSO}
\newcommand{\linselecttext}{LASSO-LinSelect}
\newcommand{\genconh}{h}
\usepackage{pdfsync}
\usepackage{fancyhdr}

\usepackage{color}
\usepackage{xcolor}

\newif\ifgray
\graytrue
\grayfalse
\IfFileExists{BackgroundColor.tex}{\grayfalse
}{}
\IfFileExists{../BackgroundColor.tex}{}{}
\ifgray
\pagecolor{black!30!white}
\fi

\begin{document}

\title{Prediction Error Bounds for\\Linear Regression With the TREX}



\author{Jacob Bien \and Irina Gaynanova \and Johannes Lederer${}^*$ \and Christian M\"uller
}

\authorrunning{Bien, Gaynanova, Lederer, M\"uller} 

\institute{${}^*$Johannes Lederer\at 
Departments of Statistics and Biostatistics\\
University of Washington\\
Box 354322\\
Seattle, WA 98195-4322\\
\email{ledererj@uw.edu}}

\ifarxiv
\date{}
\else
\date{Received: date / Accepted: date}
\fi

\maketitle

\begin{abstract}
The TREX is a recently introduced approach to sparse linear regression. In contrast to most well-known approaches to penalized regression, the TREX can be formulated without the use of tuning parameters.  
In this paper, we establish the first known prediction error bounds for the TREX. Additionally, we introduce extensions of the TREX to a more general class of penalties, and we provide a bound on the prediction error in this generalized setting. These results deepen the understanding of TREX from a theoretical perspective and provide new insights into penalized regression in general.
\keywords{TREX \and high-dimensional regression \and tuning parameters \and oracle inequalities}
\subclass{62J07}
\end{abstract}

\newcommand{\citep}[1]{\cite{#1}}
\newcommand{\citet}[1]{\cite{#1}}
\newcommand{\citealt}[1]{\cite{#1}}
\newcommand{\newjl}[1]{\color{blue}#1\color{black}}
\ifarxiv
\renewcommand{\newjl}[1]{\color{black}#1\color{black}}
\fi

\newcommand{\newjls}[1]{\color{olive}#1\color{black}}
\renewcommand{\newjls}[1]{#1}
\newcommand{\sqrtlassotext}{square-root LASSO}
\newcommand{\lassotext}{LASSO}

\section{Introduction}
The high dimensionality of contemporary datasets has fueled extensive work in developing and studying penalized estimators, such as the LASSO~\citep{Tibshirani-LASSO}, SCAD~\citep{Fan_Li01}, and MCP~\citep{Zhang10}. However, the performance of each of these estimators depends on one or more tuning parameters that can be difficult to calibrate. In particular, among the many different approaches to the calibration of tuning parameters, only very few are equipped with theoretical guarantees.

The TREX~\citep{LedererMueller:14} is a recently proposed approach to sparse regression that attempts to sidestep this calibration issue by avoiding tuning parameters altogether. Numerical studies have suggested that it may be a competitive basis for high-dimensional linear regression and graph estimation~\citep{LedererMueller:14}. Moreover, the TREX fits within the knockoff framework of~\citet{Barber2014a}, leading to theoretical bounds for false discovery rates~\citep{Bien16}. 
On the other hand, the TREX has not yet been equipped with any oracle inequalities, which form the foundation of high-dimensional theories.


In this paper, we make two contributions.
\begin{itemize}
\item We develop the first prediction error bounds for linear regression with the TREX.
\item We extend the TREX to a class of more general penalties.
\end{itemize}
These results represent a first step toward a complete theoretical treatment of the TREX approach. 
In addition, the results 
highlight the role of tuning parameters in high-dimensional regression in general.  For example, while most results in the literature make statements about an estimator assuming an oracle could choose the tuning parameter, the prediction error bounds we present pertain to the TREX as it would actually be used by a practitioner. Finally, despite its non-convex objective function, recent work has shown that the TREX can be efficiently globally optimized using machinery from convex optimization~\citep{Bien16}. 
This property plays a key role in closing a practical gap that would otherwise be present in any theoretical results that are based on the global minimizer.

The rest of the paper is organized as follows. In the remainder of this section, we specify the framework and mention further literature. In Section~\ref{sec:review}, we review the TREX approach. In Section~\ref{sec:theory}, we present the theoretical guarantees. In Section~\ref{sec:generalization}, we formulate an extension of the TREX to arbitrary norm penalties and show how the theory presented earlier generalizes to this context. In Section~\ref{sec:discussion}, we conclude with a discussion. The proofs are deferred to the Appendix.


\newcommand{\noise}{\varepsilon}
\newcommand{\noisesep}{\zeta}
\subsection*{Framework and Notation}
We consider linear regression models of the form
\begin{equation}
  \label{d:model}
  Y= X\bt + \noise,
\end{equation}
where $Y\in\rn$ is a response vector, $X\in\R^{n\times p}$ a design matrix, and $\noise\in\rn$ a noise vector. We allow in particular for high-dimensional settings, where $p$ rivals or exceeds~$n$, and general distributions of the noise $\noise$ that do not need to be known to the statistician. Typical targets in  this framework are $\bt$ (estimation), the support of $\bt$ (variable selection), or $X\bt$ (prediction). In this study, we focus on prediction. 

\newjl{Throughout the manuscript,  we use the norms $\|v\|_1 := \sum_{i=1}^l|v_i|$, $\|v\|_2 := (\sum_{i=1}^lv_i^2)^{1/2}$,  and $\|v\|_{\infty} := \max_{i\in\{1,\dots,l\}} |v_i|$ for vectors $v\in \R^l$, $l\in\{1,2,\dots\}$.}

\subsection*{Related Literature}
While this work is focused exclusively on the TREX, for completeness, we mention some alternative approaches to tuning parameter calibration. Calibration schemes for the LASSO have been introduced and studied in various papers, including~\citet{Chatterjee15,Chetelat_Lederer_Salmon14,Chichignoud_Lederer_Wainwright14,Giraud12,Meinshausen:2010,Sabourin15,Shah13}. A LASSO-type algorithm with variable selection guarantees was introduced in~\citet{Lim16}. 
\newjls{The square-root LASSO~\citep{sqrtlasso}, described in the next section, and similarly the scaled LASSO~\citep{scaledlasso} are aimed at avoiding the calibration to the noise variance. }
 A version of the square-root LASSO with a group penalty was formulated in~\citet{Yoyo13}.

\section{A Brief Review of the TREX Approach}\label{sec:review}

The TREX was motivated by a reformulation and extension of the square-root LASSO approach. Following the line of arguments in~\citet{LedererMueller:14}, we first review the LASSO estimator~\citep{Tibshirani-LASSO} defined as
\begin{equation}
\label{d:lasso}
  \lassoa\in\argmin_{\beta\in\rp}\Big\{{\|Y-X\beta\|_2^2}+2\lambda\|\beta\|_1\Big\}~~~~~~~~~~~~~~~(\lambda > 0).
\end{equation}
\newjl{It is well-understood that tuning parameters of the form
\begin{equation*}
\lambda=\genconh {\|X^\top \noise\|_\infty}= \genconh{\sigma\|X^\top \noisesep\|_\infty} ,
\end{equation*}
where $\genconh>0$ is a constant, and where we factor out the standard deviation of the noise $\sigma$ (that is, $\noisesep:=\noise/\sigma$) for illustration purposes, can be theoretically sound---see~\citealt{Buhlmann11} and references therein.  Indeed, these choices for $\lambda$ follow from oracle inequalities for the LASSO (see the next section). 
In practice, however, the corresponding calibration of~$\lambda$ can be challenging, because $\genconh{\sigma\|X^\top \noisesep\|_\infty}$ is a random quantity that  depends on several aspects of the model, including the design matrix~$X$ and the in practice unknown  standard deviation of the noise~$\sigma$ and noise vector~$\noisesep$.}


The square-root LASSO~\citep{sqrtlasso} is designed to obviate the need for calibration with respect to the standard deviation of the noise. As described in~\citet{LedererMueller:14}, one can write the square-root LASSO as
\begin{equation}
\label{SQRT2}
  \lassosqrta\in\argmin_{\beta\in\rp}\Bigg\{\frac{{\|Y-X\beta\|_2^2}}{\frac{\|Y-X\beta\|_2}{\sqrt n}}+\gamma\|\beta\|_1\Bigg\}~~~~~~~~~~~~~~~(\gamma>0) \,,
\end{equation}
which highlights the view of this as the LASSO with a scaling factor of ${\|Y-X\beta\|_2}/{\sqrt n}$. This  factor acts as an inherent estimator of the standard deviation of the noise $\sigma$ and therefore makes further calibration of $\sigma$ unnecessary.
\newjl{A suitable theoretical form of $\gamma$ is (see, for example,~\citealt{sqrtlasso})
\begin{equation*}
\gamma=\genconh{\|X^\top \noisesep\|_\infty}
\end{equation*}
for constants $\genconh>0$.
Note that $\genconh{\|X^\top \noisesep\|_\infty}$ is still a random quantity that depends on the typically unknown $\noisesep$, so  that  $\gamma$ is still subject to a tuning scheme.}  

The motivation for the TREX is to automatically calibrate all model parameters. The key idea is to estimate the entire quantity ${\sigma\|X^\top \noisesep\|_\infty}$ rather than $\sigma$ alone. Thus, in line with the discussion above, ${\sigma\|X^\top \noisesep\|_\infty}$ is mimicked by \newjls{${\|X^\top(Y-X\beta)\|_\infty}$}, leading to the estimator
\begin{align}
\label{trexobjective}
  \trex\in\argmin_{\beta\in\rp}\Big\{\frac{\|Y-X\beta\|_2^2}{c\|X^\top (Y-X\beta)\|_\infty}+\|\beta\|_1\Big\} \,,
\end{align}
where $c \in (0,2)$ is a constant. The value $c=1/2$ is used in the original TREX proposal~\citep{LedererMueller:14}. With this setting, TREX is a single, tuning-free estimator. Similar to the bootstrap LASSO \citep{Bach08}, the TREX can be equipped with bootstrapping techniques for variable ranking \citep{LedererMueller:14} or used for graphical model learning using node-wise regression \citep{LedererMueller14b}.  

An initial practical challenge regarding the TREX estimator was the non-convexity of the underlying objective function in~\eqref{trexobjective}. While first approaches to minimizing the TREX objective function relied on fast heuristic approximation schemes~\citep{LedererMueller:14}, it was shown in \citet{Bien16} that the globally optimal solution of the TREX objective can be found in polynomial time via \textit{convex} optimization techniques. This was achieved by stating the problem in~\eqref{trexobjective} as an equivalent minimization problem over $2p$ convex sub-problems. These sub-problems can be solved via \newjl{second-order cone } programming \citep{Bien16} or proximal algorithms thanks to the existence of efficient proximity operators \citep{Combettes16}. In the $n>p$ setting, availability of the globally optimal solution and the different convex sub-problem solutions led to provable guarantees on false discovery rates of the TREX~\citep{Bien16} by extending the knock-off filter framework originally introduced in \citet{Barber2014a}. Other theoretical guarantees for the TREX have, thus far, remained elusive.

\newcommand{\xte}{\|X^\top \varepsilon\|_\infty}
\section{Theoretical Guarantees for the TREX}\label{sec:theory}
Here, we derive novel theoretical guarantees for sparse linear regression with the TREX. We focus on guarantees on the prediction loss \newjls{$\|X\trex-X\betatrue\|^2_2/n$, }  which is an important step toward a complete understanding of the theoretical properties of the TREX. We use the LASSO as a point of reference, since it is the most well-known and well-studied approach for sparse linear regression. 
\newjl{For ease of exposition, we consider fixed design matrices $X$ that are normalized such that each column has Euclidean norm~$\sqrt n$. } 
\newjls{We also do not impose restrictions on the noise distribution; }
 in particular, we allow the entries of the noise vector to be correlated, the noise vector~$\varepsilon$ and the design matrix~$X$ to be correlated, and we allow for heavy-tailed noise distributions. The proofs of our results are deferred to the Appendix.

\pagebreak
\subsection{Review of Bounds for \newjl{LASSO-type Estimators}}\label{s:lassor}
\newjls{For the \lassotext\ estimator defined in~\eqref{d:lasso}, two types of prediction bounds are known in the literature. 
The first type is  termed a ``fast-rate bound.'' 
A standard representative invokes the compatibility condition} 
\begin{equation}\label{compcond}
  \nu\le \frac{\sqrt s\|X\eta\|_2}{\sqrt n\|\eta_S\|_1}\text{~~for all~~}\eta\in\rp\text{~such that~}\|\eta_{S^c}\|_1\leq 3\|\eta_S\|_1,
\end{equation}
\newjls{where $\nu>0$ is a constant, } $S:=\operatorname{supp}(\bt)$ is the support of $\bt$ (the index set of the non-zero entries of $\bt$), $S^c$ the complement of $S$, and $s:=|S|$ the sparsity level. The compatibility condition ensures that the correlations in the design matrix~$X$ are small. Under this assumption, the prediction loss can be bounded as follows, cf. \cite[Theorem~6.1]{Buhlmann11}.
\begin{lemma}\label{fastratelasso}
If $\lambda\geq 2\|X^\top\varepsilon\|_\infty$ and the compatibility condition~\eqref{compcond} is met for a constant $\nu>0$, the prediction loss of the LASSO satisfies
  \begin{equation*}
    \frac{\|X\lassoa-X\bt \|_2^2}{n}\leq \frac{16s \lambda^2 }{\nu^2n^2}.
  \end{equation*}
\end{lemma} 
Similar bounds for the lasso can be found in\newjls{~\citet{Bickel09,Bunea06,Candes09,Candes07,ArnakMoYo14,DT07,DT12a,DT12b,vandeG07,Sara09,Raskutti10,RigTsy11}. } \newjl{ Note that we have adopted the normalization $1/n$ of the prediction error throughout, relating to the $\sqrt n$-normalization of the columns of $X$. 
If one would know the true support of $\betatrue$, then a least-squares on this support would have expected prediction loss~$s\sigma^2/n$. 
The above result states that the lasso can---under the given assumptions---achieve the same rate up to constants and an ``entropy'' factor encapsulated in~$\lambda^2$. Indeed, if $\lambda=2\|X^\top\varepsilon\|_\infty$, we recover the rate $s\sigma^2/n$ up to  the constant $16/\nu^2$ and the entropy factor $\|X^\top\varepsilon\|_\infty^2/(\sigma^2n)$.  
Given a distribution, one can bound the latter factor in probability by using maximal inequalities, see~\cite{Boucheron13,vandeGeer00,vanderVaart96} or more recently~\cite{vdGeer11b,YoyoSara12,wellner2017bennett}. In the case of i.i.d. Gaussian noise, for example, we find that the factor can be bounded by $\genconh\log p$ with a sufficiently large constant $\genconh$.}

\newjl{Such bounds can also be derived for the \sqrtlassotext\ and other penalized regression methods. 
For example, using  the techniques in \cite{sqrtlasso}, one can derive (under some minor additional conditions) that there are constants $\genconh,\genconh'>0$ such that if $\gamma\geq \genconh\|X^\top\noisesep\|_\infty$, the prediction loss of the \sqrtlassotext\ in~\eqref{SQRT2} satisfies
  \begin{equation*}
    \frac{\|X\lassosqrta-X\bt \|_2^2}{n}\leq \frac{\genconh' s\sigma^2 \gamma^2 }{\nu^2n^2}.
  \end{equation*}
The same bound (except for slightly different assumptions and constants) holds also for \lassotext-LinSelect~\cite{Baraud09}, for example. We refer the reader to \cite[Proposition~4.2 and~4.3]{Giraud12} for detailed prediction bounds for the \sqrtlassotext\ and \lassotext-LinSelect.
}

\newjls{The second type of bound is termed a ``slow-rate bound.''  In strong contrast to the ``fast-rate bounds,'' } these bounds provide guarantees on the prediction error without assuming weakly correlated~$X$. The standard representative for this type of bound is as follows \citep{HCB08,Kolt10,MM11,RigTsy11}.
\begin{lemma}\label{slowratelasso}
  If $\lambda\geq \|X^\top\varepsilon\|_\infty$, the prediction loss of the LASSO satisfies the following bound:
  \begin{equation*}
   \frac{ \|X\lassoa-X\bt \|_2^2}{n}\leq \frac{4\lambda \|\bt\|_1}{n}.\label{eq:slow}
  \end{equation*}
\end{lemma}
\newjl{The term ``slow-rate'' comes from the fact that the representative above provides rates typically not faster than $1/\sqrt n$ (as compared to $1/n$ above).
Indeed, if again $\lambda=2\|X^\top\varepsilon\|_\infty$ and if the noise is again i.i.d. Gaussian, standard maximal inequalities yield the bound $\genconh\sigma\sqrt{\log p}\|\bt\|_1/\sqrt n$ for a constant $\genconh>0$.
However, there are three main reason for why a superficial comparison of rates is not sufficient and the term ``slow-rate'' is misleading:
First, the assumptions on~$X$ needed for ``fast-rate bounds'' (such as the compatibility condition invoked above) can typically not be verified in practice, because they involve quantities (such as the support of~$\bt$) that are unknown; ``slow-rate bounds,'' in contrast, hold for any~$X$. Second, fast ``slow-rate bounds'' (with rate close to $1/n$) can be derived in some settings~\citep{ArnakMoYo14}, and the constants in ``fast-rate'' bounds can be unfavorably large, masking the $1/n$ rate. 
Third, the rate $1/\sqrt n$ cannot be improved altogether unless further assumptions are imposed~\cite[Proposition~4 on Page~561]{ArnakMoYo14}. }
Refined versions of both types of bound and more exhaustive comparisons can be found in~\citet{ArnakMoYo14,vdGeer11,YoyoMomo12,Lederer16b,Zhuang17}.

\newjl{Similarly for the \sqrtlassotext, if $\gamma=\genconh \|X^\top\noisesep\|_\infty$ for a suitable $\genconh>0$, then the prediction loss  satisfies the following bound:
  \begin{equation*}
   \frac{ \|X\lassosqrta-X\bt \|_2^2}{n}\leq \frac{\genconh'\sigma\gamma \|\bt\|_1}{n}
  \end{equation*}
for a constant $\genconh'>0$.
We refer to \cite{Lederer16b} for details and further references.}

A major limitation of both types of result is the condition on the tuning parameter~$\lambda$. As indicated in Lemmas~\ref{fastratelasso} and~\ref{slowratelasso}, the tuning parameter (i)~needs to be large enough to satisfy the given assumptions but (ii)~small enough to provide sharp bounds. Since \newjls{$\|X^\top\varepsilon\|_\infty$ \newjl{(and equivalently $\|X^\top\noisesep\|_\infty$ in the \sqrtlassotext\ case) are }} unknown in practice, one tries to establish a trade-off by using a data-dependent tuning parameter~$\lest\equiv\lest(Y,X)$ when it comes to applications. 
Plugging this data-dependent tuning parameter  in Lemmas~\ref{fastratelasso} and~\ref{slowratelasso} yields the following.
\begin{corollary}\label{rateslassotuningfast} If $\lest\geq 2\|X^\top\varepsilon\|_\infty$ and the compatibility condition~\eqref{compcond} is met for a constant $\nu>0$, the prediction loss of the LASSO satisfies
  \begin{equation*}
    \frac{\|X\lasso(\lest)-X\bt \|_2^2}{n}\leq \frac{16s {\lest}^2 }{\nu^2n^2}.
  \end{equation*}
\end{corollary}
\begin{corollary}\label{rateslassotuningslow}
  If $\lest\geq \|X^\top\varepsilon\|_\infty$, the prediction loss of the LASSO satisfies
  \begin{equation*}
   \frac{ \|X\lasso(\lest)-X\bt \|_2^2}{n}\leq \frac{4\lest \|\bt\|_1}{n}.
  \end{equation*}
\end{corollary}
These bounds now hinge on guarantees for~$\lest(Y,X)$.  In particular, for such bounds to hold, one would need to guarantee that $\lest(Y,X)\geq 2\|X^\top\varepsilon\|_\infty$; however, standard data-based calibration schemes such as cross-validation, BIC, and AIC, lack such results.
To our knowledge, there are neither results showing that~$\lest$ is sufficiently large to satisfy the given assumptions, nor that~$\lest$ is sufficiently small to provide reasonable bounds. In this sense, the above guarantees are of academic value only.

\subsection{Bounds for the TREX}\label{s:trex_t}
\newcommand{\lambdatrex}{\lest}

\newjls{In this section, we derive two types of prediction bounds for the TREX. The first bound (Theorem~\ref{trexfast}) expresses the TREX's prediction error in terms of the LASSO's prediction error, whereas the second bound (Theorem~\ref{slowratetrex}) stands on its own.}


\newjls{Functions of the form $\|X^{\top}(Y-X\cdot)\|_{\infty}$ are well-known in LASSO settings, since the LASSO KKT conditions specify that $\lambda = \|X^{\top}(Y-X\lassoa\|_{\infty}$ as long as $\lassoa\neq 0$ (and more generally specify that $\|X^{\top}(Y-X\lassoa)\|_{\infty}\le\lambda$). 
We state the TREX bounds in terms of the corresponding quantity, which we denote as 
\begin{equation}\label{eq:hatu}
\hat u\equiv \hat u(Y,X):=\|X^\top(Y-X\trex)\|_\infty.
\end{equation}
}

First, we establish a direct connection between the prediction performance of the TREX and the LASSO under the following assumption.
\begin{assumption}\label{signalstrengthup}
The regression vector $\bt$ is sufficiently small such that
 \begin{align*}
\|\bt\|_1\le\frac{\|\varepsilon\|_2^2}{16\|X^{\top}\varepsilon\|_{\infty}}.
\end{align*} 
\end{assumption}
\newjl{In the case of i.i.d. Gaussian noise, this assumption 
can be loosely interpreted as saying that $\|\bt\|_1$ should be bounded by $\sigma\sqrt{n/\log p}$, which holds in the standard high-dimensional limit in which $(\log p) /n\to0$.
This assumption might look unusual, but recall that the classical ``slow-rate bounds'' for the \lassotext\ have a similar assumption implicit in their formulation. 
In Lemma~\ref{slowratelasso}, for example, under i.i.d. Gaussian noise and optimal tuning, the prediction bound is $\genconh \sigma\sqrt{(\log p)/n}\|\bt\|_1$ for a constant $\genconh$, which is a non-trivial bound only if  $\|\bt\|_1\ll \sigma\sqrt{n/\log p}$.}

\newjls{
\begin{theorem}\label{trexfast} Let Assumption~\ref{signalstrengthup} be fulfilled and let $\lambdatrex:=\max\{2 \hat u, 8\|X^\top \varepsilon\|_\infty/c\}$, where \newjl{$c\in (0,2)$ and } $\hat u$ is defined in~\eqref{eq:hatu}. If $\hat u\le \|X^\top Y\|_\infty/2$, then the prediction loss of the TREX satisfies
\begin{equation*}
\frac{\|X\trex-X\bt\|^{2}_{2}}{n}\le \frac{3\|X\lassoat-X\bt\|^{2}_{2}}{4n}+\frac{7\|X^\top\varepsilon\|_\infty\|\lassoat-\bt\|_1}{2n}.
\end{equation*}
\end{theorem}
}
 The quantity $\hat u$ from~\eqref{eq:hatu} plays a similar role for the TREX as the data-driven tuning parameter $\lest$ plays for the LASSO. The main feature of Theorem~\ref{trexfast}, however, is that it allows for any value  $\hat u\le \|X^\top Y\|_\infty.$ Indeed, the tuning parameter $\lambdatrex$ is lower bounded by $8\xte/c$, irrespective of the value of $\hat u$, which ensures that the condition $\lambda\gtrsim \xte$ on the tuning parameter $\lambda$ in standard results for the LASSO, such as Corollaries~\ref{rateslassotuningfast} and~\ref{rateslassotuningslow}, is always fulfilled. Therefore, Theorem~\ref{trexfast} in conjunction with the known results for LASSO,  provides a concrete bound for the prediction loss $\|X\trex-X\bt\|^{2}_{2}/n$ of the TREX. 
\newjl{Specifically, if the \lassotext\ satisfies the compatibility condition and $\hat u$ is of the same order as $\|X^\top \varepsilon\|_\infty$, the above bound implies the ``fast-rate'' $s\sigma^2(\log p)/n$ up to constants and an entropy factor as before.
The following corollary highlights these aspects once more.}

\begin{corollary}\label{cor:fastT} Let Assumption~\ref{signalstrengthup} be fulfilled, and let $\lambdatrex:=\max\{2 \hat u, 8\|X^\top \varepsilon\|_\infty/c\}$, where $\hat u$ is defined in~\eqref{eq:hatu}. If $\hat u\le \|X^\top Y\|_\infty/2$, then the prediction loss of the TREX satisfies
\begin{equation*}
\frac{\|X\trex-X\bt\|^{2}_{2}}{n}\le \frac{12s \lambdatrex^2}{\nu^2n^2},
\end{equation*}
where $s$ is the sparsity level of $\betatrue$, and $\nu$ is the compatibility constant from~\eqref{compcond}.
\end{corollary}
This corollary clearly resembles Corollary~\ref{rateslassotuningfast} for the LASSO, yet a significant distinction is that for the LASSO result to be applicable, one would also need a result showing that the data-driven choice of the tuning parameter exceeds $2\|X^\top\varepsilon\|_\infty$.  By contrast, Corollary~\ref{cor:fastT} holds for the TREX estimator exactly as it would be used in practice.  Whether Corollary~\ref{cor:fastT} represents a tight bound for the prediction error of the TREX is another matter:  when $2 \hat u\le 8\|X^\top \varepsilon\|_\infty/c$, the bound behaves like the bound in Corollary~\ref{rateslassotuningfast} with the best choice of $\tilde\lambda$; when instead $2 \hat u> 8\|X^\top \varepsilon\|_\infty/c$, the bound is not necessarily tight. We discuss this in more detail further below.

Next, we bound the prediction loss of the TREX directly. Specifically, we are interested in deriving a so-called ``slow-rate bound,'' which allows for potentially large correlations in~$X$ (see discussion in Section~\ref{s:lassor}). To derive this bound, we use a different assumption.
\begin{assumption}\label{signalstrengthlow} The regression vector $\bt$ is sufficiently large such that
    \begin{equation*}
  \|X^{\top}X\bt\|_{\infty}\ge \Big(1+ \frac{2}{\consttrex}\Big)\|X^{\top}\varepsilon\|_{\infty}.    
    \end{equation*}
 \end{assumption}
 Note that this assumption implies $\|X^{\top}Y\|_{\infty}\ge 2\|X^{\top}\varepsilon\|_{\infty}/\consttrex$ via the triangle inequality.  
\newjl{In the orthogonal case, in which $X^{\top}X=n\operatorname{I}_n$, it holds that $\|X^{\top}X\bt\|_{\infty}=n\|\bt\|_\infty;$ more generally, under the $\ell_\infty$-restricted eigenvalue condition~\cite[Equation~(7) on Page~6]{Chichignoud_Lederer_Wainwright14}, it holds that $\|X^{\top}X\bt\|_{\infty}\geq g n\|\bt\|_\infty$ for some constant~$g$. Then, in the case of i.i.d. Gaussian noise, a sufficient condition for Assumption~\ref{signalstrengthlow} is  that $\|\bt\|_\infty$ of larger order in $n$ than $\sigma\sqrt{(\log p)/n}$.
The assumption can, therefore, be interpreted as a condition on the minimal signal strength, and would naturally hold in the standard high-dimensional limit of $(\log p)/n\to 0$.}

\begin{theorem}\label{slowratetrex} Let Assumption~\ref{signalstrengthlow} be fulfilled, and let $\hat u$ be as in~\eqref{eq:hatu}. If $\hat u\le \|X^\top Y\|_\infty$, then the prediction loss of the TREX satisfies
\begin{equation*}
\frac{\|X\trex-X\betatrue\|_{2}^{2}}{n}\le \frac{\left(2\|X^{\top}\varepsilon\|_{\infty}+\max\left\{\hat u,2\|X^{\top}\varepsilon\|_{\infty}/c\right\}\right)\|\betatrue\|_{1}}{n}.
\end{equation*}
\end{theorem}
\newjl{This result is closely related to Corollary~\ref{rateslassotuningslow}. 
In particular, since Corollary~\ref{rateslassotuningslow} requires $\lest\geq \xte$, the right-hand sides in the bounds of Theorem~\ref{slowratetrex} and of Corollary~\ref{rateslassotuningslow} resemble each other closely\footnote{\footnotesize The right-hand side in Corollary~\ref{rateslassotuningslow} has a minimal magnitude of $4\xte \|\bt\|_1/n$, while the right-hand side in Theorem~\ref{slowratetrex} has a minimal magnitude of $(2+2/c)\xte \|\bt\|_1/n$.}.
The rates in both results are, as discussed, ``slow-rates'' of order at least $1/\sqrt n$.  
The major difference between the two results are the assumptions on $\lest$ and $\hat u$:
While the assumption $\lest\geq \|X^\top\varepsilon\|_\infty$ depends on quantities that are unknown in practice, the assumption $\hat u\le \|X^\top Y\|_\infty$ depends on the data only. This means that in practice, one can check if $\hat u$ satisfies the requirements. Alternatively, one can add the convex constraint~$\|X^\top (Y-X\beta)\|_\infty\le \|X^\top Y\|_\infty$ directly into the optimization procedure of the TREX, which automatically ensures---without changing the proofs---that $\hat u\le \|X^\top Y\|_\infty$. 
Having verifiable assumptions is an advantage of Theorem~\ref{slowratetrex} over Corollary~\ref{rateslassotuningslow}; however, a price paid is that the bound in Theorem~\ref{slowratetrex} may be large if $\hat u\gg \|X^\top\varepsilon\|_\infty$.
 Note finally that as both results are oracle inequalities, the bounds themselves are not known in practice.}

We now discuss the conditions imposed by Assumptions~\ref{signalstrengthup} and~\ref{signalstrengthlow}. While each of the derived prediction bounds relies on only one of these assumptions, it is of interest to see whether these assumptions can hold at the same time. Note that for standard noise distributions, it holds that $\|\varepsilon\|_2^2$ is of order~$n$, and classical entropy bounds ensure that $\|X^{\top}\varepsilon\|_{\infty}$ is of order \newjl{$\sigma\sqrt{n\log p}$ } (recall the normalization of~$X$). Moreover, \newjl{for near-orthogonal design (in the sense of the $\ell_\infty$-restricted eigenvalue, for example), it holds that $\|X^{\top}X\bt\|_{\infty}$ is lower bounded by $n\|\bt\|_\infty$ up to constants. As such, at a high level, Assumptions~\ref{signalstrengthup} and~\ref{signalstrengthlow} together imply $\sigma\sqrt{(\log p)/n}/f\leq \|\bt\|_\infty\leq\|\bt\|_1\leq f \sigma\sqrt{n/\log p}$ for a constant $f$, } which are mild conditions on the signal~$\bt$ when $n$ is sufficiently large. For small~$n$, however, the constants in Assumptions~\ref{signalstrengthup} and~\ref{signalstrengthlow} become relevant and can be too large. Conditions with smaller constants can be found in the Appendix, and we expect that the constants can be further decreased. However, the key point is that the restrictions imposed on the model are very mild if $n$ is reasonably large relative to $\log p$. Still, an open conceptual question is whether it is possible to develop theories (for any estimator!) that do without assumptions on the signal strength altogether.

A more stringent limitation of the above inequalities is the dependence on~$\hat u.$ Since tight upper bounds for~$\hat u$ are currently lacking, it is not clear if the inequalities provide an optimal control of the prediction errors. However, we argue that our results still improve the existing theory for sparse linear regression substantially. We first note that  known LASSO prediction bounds do not involve estimated quantities only because these bounds do not address tuning parameter calibration altogether. When the LASSO is combined with a calibration scheme, the bounds do in fact involve  estimated tuning parameters (see Corollaries 3 and 4) that have eluded any statistical analysis so far. Our bounds for the TREX provide a better control because they involve the maximum of $\hat u$ and the optimal bound, thus providing a lower bound on the relevant quantities. 

\newjl{Similar comments apply to the comparison of our TREX theory with the theories of the \sqrtlassotext/\scaledlassotext\ and \linselecttext:
since the three latter methods all invoke (functions of) the normalized noise $\noisesep$ in their tuning schemes, they require additional calibration unless (the distribution of)~$\noisesep$ is known.
However, when sharp tail bounds for (the appropriate functions of) $\noisesep$ are known, these methods are currently equipped with more comprehensive theories than what is presented here (and in the \lassotext-literature),  and unlike the results here, they then also allow for bounds in probability~\cite[Proposition~4.2 and~4.3]{Giraud12}. 
}

\newjl{As for calibration schemes, we mention cross-validation as the most popular one. Standard $k$-fold cross-validation can provide unbiased estimation of a prediction risk~\cite[Page~57]{Arlot10}. However, this hinges on i.i.d. data and the risk is in out-of-sample prediction (as compared to the in-sample prediction considered here) for a sample size smaller than~$n$ (since cross-validation involves data splitting). A comprehensive overview of cross-validation theory can be found in~\cite{Arlot10}, further ideas in~\cite{Arlot11} and others. Nevertheless,  there are currently no non-asymptotic guarantees for \lassotext-type estimators calibrated by $k$-fold cross-validation.}

In summary, the presented prediction bounds for the TREX avoid the problem of ``too small'' tuning parameters that is present for \newjls{\lassotext-type } bounds (see Section~\ref{s:lassor}). Thus, the TREX prediction bounds provide an advancement over the  theory for the LASSO (and similarly, the \sqrtlassotext, the scaled LASSO, MCP, and other penalized methods) combined with cross-validation, information criteria, or other standard means for tuning parameter selection. However, the inequalities do not solve the problem of ``too large'' tuning parameters, yet. \newjl{Being able to bound $\hat u$ in terms of $\|X^\top \varepsilon\|_\infty$, for example, would allow us to extend our results in this direction. } Thus, our results provide a not yet complete but certainly improved theoretical framework for sparse linear regression.    

\section{Generalizations to General Norm Penalties}\label{sec:generalization} \label{sec:beyond-l1-norm}
The $\ell_1$-norm in the penalty of the LASSO and the TREX reflects the assumption that the regression vector is sparse. However, many applications involve more complex structures for~$\bt$. In this section, we discuss generalizations of the TREX objective function to general norm penalities. For this, we consider general norms $\Omega$ on $\R^p$ with corresponding dual norms defined by $\Omega^*(\eta):=\sup\{\eta^\top\beta: \Omega(\beta)\le 1\}$. A generalized version of the TREX can then be formulated as
\begin{align}
  \min_{\beta\in\rp}\left\{\frac{\|Y-X\beta\|^2_2}{\consttrex\Omega^*(X^\top(Y-X\beta))}+\Omega(\beta)\right\},\label{eq:genTREX}
\end{align}
assuming that a minimum exists.  An example is the {\em group TREX} that uses the norm penalty $\Omega(\beta):=\sum_{G\in\mathcal G}w_G\|\beta_G\|_2$   for a given partition $\mathcal G$ of $\{1,\ldots,p\}$ and given weights $w_1,\dots,w_{|\mathcal G|}>0.$ Explicitly, we define the group TREX as
\begin{align*}
  \min_{\beta\in\rp}\left\{\frac{\|Y-X\beta\|^2_2}{\consttrex\max_{G\in\mathcal G}\{\|X_G^\top(Y-X\beta)\|_2/w_G\}}+\sum_{G\in\mathcal G}w_G\|\beta_G\|_2\right\}.
\end{align*}
The group TREX is the TREX analog to the well-known group LASSO \citep{YuanLin06} and group square-root LASSO \citep{Yoyo13}.  In the special case in which all groups are of size one, that is, $\mathcal G=\{1,\dots,p\}$, the group TREX reduces to a weighted version of the standard TREX problem: 
\begin{align*}
  \min_{\beta\in\rp}\left\{\frac{\|Y-X\beta\|^2_2}{\consttrex\max_{j\in\{1,\dots,p\}}\{|X_j^\top(Y-X\beta)|/w_j\}}+\sum_{j\in\{1,\dots,p\}}w_j|\beta_j|\right\}.
\end{align*}
%
The weighting can be further extended to allow for completely unpenalized coordinates.  Unpenalized predictors arise commonly in practice when one wants to allow for an  intercept or has prior knowledge that certain predictors should be included in the fitted model. To incorporate unpenalized predictors in the TREX, we denote the set of indices corresponding to the unpenalized predictors by $\mathcal U\subseteq\{1,\ldots,p\}$ and the complementary set by $\mathcal P:=\{1,\ldots,p\}\setminus\mathcal U$. Let us, with some abuse of notation, denote the generalized TREX by~$\trex$. The unpenalized predictors are determined by forcing the estimator~$\trex$ to satisfy
\begin{equation*}\label{eq:unpenalized}
  X^\top_{\mathcal U}(Y-X\est)=0.
\end{equation*}
The subscript indicates that only the columns of a matrix (or the entries of a vector) in the corresponding set are considered. In the special case of an intercept, where $U=\{1\}$ and $x_1=(1,\dots,1)^\top$, this reads
\begin{equation*}
  (1,\dots,1)^\top(Y-X\est)=0.
\end{equation*}
This is the identical constraint for the standard LASSO with intercept. More generally, \eqref{eq:unpenalized} implies
\begin{align*}
  \est_{\mathcal U}&=(X^\top_{\mathcal U}X_{\mathcal U})^{+}X^\top_{\mathcal U}(Y-X_{\mathcal P}\est_{\mathcal P}),
\end{align*}
where $(X^\top_{\mathcal U}X_{\mathcal U})^{+}$ is a generalized inverse of $X^\top_{\mathcal U}X_{\mathcal U}$. Incorporating this constraint in the TREX criterion yields the estimator
\begin{align}\label{trexprior}
  \begin{split}
&\est_{\mathcal P}:= \min_{\beta_{\mathcal P}\in\R^{|\mathcal P|}}\left\{\frac{\|M_{\mathcal U}(Y-X_{\mathcal P}\beta_{\mathcal P})\|^2_2}{\consttrex\Omega^*(X_{\mathcal P}^\top M_{\mathcal U}(Y-X_{\mathcal P}\beta_{\mathcal P}))}+\Omega(\beta_{\mathcal P})\right\}\\
&\est_{\mathcal U}:=(X^\top_{\mathcal U}X_{\mathcal U})^{+}X^\top_{\mathcal U}(Y-X_{\mathcal P}\est_{\mathcal P}),
  \end{split}
\end{align}
where $M_{\mathcal U}:=(\operatorname{I}-X_{\mathcal U}(X^\top_{\mathcal U}X_{\mathcal U})^{+}X^\top_{\mathcal U})$ is the projection on the orthogonal complement of the span of the columns with indices in $\mathcal U$. Note that the estimator~\eqref{trexprior} simplifies to the least-squares estimator if $\mathcal U=\{1,\dots, p\}$ and simplifies to  \eqref{eq:genTREX} for $\mathcal U=\emptyset$.  More generally,  $\est_{\mathcal P}$ can be considered as a (generalized) TREX on the output not explained by a least-squares of $Y$ on $X_{\mathcal U}$. Indeed, if the columns with indices in $\mathcal U$ are orthogonal to the columns with indices in $\mathcal P$, the estimator~\eqref{trexprior} simplifies to 
\begin{align*}
&\est_{\mathcal P}:= \min_{\beta_{\mathcal P}\in\R^{|\mathcal P|}}\left\{\frac{\|M_{\mathcal U}Y-X_{\mathcal P}\beta_{\mathcal P}\|^2_2}{\consttrex \Omega^*(X_{\mathcal P}^\top (M_{\mathcal U}Y-X_{\mathcal P}\beta_{\mathcal P}))}+\Omega(\beta_{\mathcal P})\right\}\\
&\est_{\mathcal U}:=(X^\top_{\mathcal U}X_{\mathcal U})^{+}X^\top_{\mathcal U}Y.
\end{align*}
In this case, the generalized TREX combines  the TREX on $(M_{\mathcal U}Y,X_{\mathcal P})$ and  least-squares estimation on~$(Y,X_\mathcal U)$. Theorem~\ref{slowratetrex} can be extended in this generalized setting.
 
  \begin{assumption}\label{signalstrengthlow_c1} The regression vector $\bt$ is sufficiently large such that
    \begin{equation*}
  \Omega^{*}(X^{\top}X\bt)\ge \left(1+\frac2{\consttrextwo}\right)\Omega^{*}(X^{\top}\varepsilon).    
    \end{equation*}
 \end{assumption}

\begin{theorem}\label{res:bound}  Let Assumption~\ref{signalstrengthlow_c1} be fulfilled and define $\hat u :=\Omega^*(X^\top(Y-X\trex))$. If $\hat u \le \Omega^*(X^\top Y)$, then the prediction loss of TREX satisfies
\begin{equation*}
\frac{\|X\trex-X\betatrue\|_{2}^{2}}{n}\le \frac{\left(2\Omega^{*}(X^{\top}\varepsilon)+\max\left\{\hat u,\frac2{\consttrextwo}\Omega^{*}(X^{\top}\varepsilon)\right\}\right)\Omega(\betatrue)}{n}.
\end{equation*}
\end{theorem} 
The bound depends on the size of the vectors $X^{\top}\varepsilon$ and $X^\top(Y-X\trex)$ as measured by the dual norm $\Omega^*$.  In the special case that $\Omega(\beta)=\|\beta\|_1$, $\Omega^*$ is simply the $\ell_\infty$-norm, and this theorem reduces to Theorem~\ref{slowratetrex}.
 We refer to Appendix~\ref{s:slowrate} for the proof details.

\section{Discussion}\label{sec:discussion}

In this paper, we have provided the first prediction error bounds for the TREX and for a newly proposed generalization of the TREX.  Our theoretical guarantees resemble known bounds for the LASSO.  However, there is an important practical difference.  The LASSO results hold only if an oracle could guarantee that a sufficiently large tuning parameter has been chosen (in practice, one uses a data-dependent calibration scheme for which such guarantees of a sufficiently large tuning parameter are not available).  By contrast, our TREX results do not have such a requirement and therefore pertain to the precise version of the estimator that is used in practice.


We envision future research in different directions. First, it would be desirable to derive sharp upper bounds for the quantity~$\hat u.$ This would lead to a complete set of guarantees for high-dimensional prediction with finitely many samples. 
\newjl{For this, one might also explore potential connections with the Dantzig selector~\cite{Candes07}, in which the quantity ${\|X^\top(Y-X\beta)\|_\infty}$ is controlled explicitly. }
Moreover, it would be of interest to study the choice of the constant~$c$ further. Numerical results in \citet{LedererMueller:14,Bien16} indicate that the fixed choice $c=1/2$ works in a variety of settings. While the presented theory covers all values of $c \in (0, 2)$, it is, however, conceivable that a data-dependent choice of~$c$ may further improve the accuracy of the TREX.
Finally, it would be interesting to extend our theoretical results to tasks beyond prediction, such as estimation and support recovery.

\newjl{\begin{acknowledgements}
We sincerely thank the editor and the reviewers for their insightful comments.
\end{acknowledgements}}

\bibliographystyle{spmpsci}      

\newjls{\bibliography{References}}


\appendix

\section{Proof of a Generalization of Theorem~\ref{trexfast}}\label{s:fastrate}
We first consider a generalization of Assumption~\ref{signalstrengthup}.
\begin{assumption}\label{signalstrengthup_kappa}
The signal $\bt$ is sufficiently small such that for some $\kappa_{1}>1$ and $\kappa_{2}>2$ \newjl{with $1/\kappa_1 + 2/\kappa_2 <1$, }
 \begin{align*}
\|\bt\|_1\le\frac14\left(1-\frac1{\kappa_{1}}-\frac2{\kappa_{2}}\right)\frac{\|\varepsilon\|_2^2}{\|X^{\top}\varepsilon\|_{\infty}}.
\end{align*} 
\end{assumption}
As a first step toward the proof of Theorem~\ref{trexfast}, we show that any TREX solution has  larger $\ell_{1}$-norm than any lasso solution with tuning parameter $\lambda=\hat u$.

\begin{lemma}\label{res:l1trexlasso} Any TREX solution \eqref{trexobjective} satisfies
$$\|\trex\|_{1}\ge\|\lasso(\hat u)\|_{1},$$
where $\lasso^{\hat u}$ is any lasso solution as in \eqref{d:lasso} with tuning parameter $\lambda=\hat u$.
\end{lemma}
\begin{proof}[of Lemma~\ref{res:l1trexlasso}] If $\lasso(\hat u)=0$, the statement holds trivially. Now for $\lasso(\hat u)\neq 0$, the KKT conditions for LASSO imply that 
  \begin{equation*}
    \|X^{\top}(Y-X\lasso(\hat u))\|_{\infty}=\hat u.
  \end{equation*}
Together with the definition of $\trex$, this yields
\begin{equation*}
\|Y-X\trex\|^{2}_{2}+\consttrextwo\hat u\|\trex\|_{1}\le\|Y-X\lasso(\hat u)\|^{2}_{2}+\consttrextwo\hat u\|\lasso(\hat u)\|_{1}.
\end{equation*}
On the other hand, the definition of the LASSO implies 
\begin{equation*}
\|Y-X\lasso(\hat u)\|^{2}_{2}+2\hat u\|\lasso(\hat u)\|_{1}\le \|Y-X\trex\|^{2}_{2}+2\hat u\|\trex\|_{1}.
\end{equation*}
Combining these two displays gives us
\begin{equation*}
(\consttrextwo-2)\hat u\|\trex\|_{1}\le (\consttrextwo-2)\hat u\|\lasso(\hat u)\|_{1}.
\end{equation*}
The claim follows now from $\consttrextwo< 2$.
\end{proof}
We are now ready to prove a generalization of Theorem~\ref{trexfast}.
\begin{theorem}\label{res:fast4}  Let Assumption~\ref{signalstrengthup_kappa} be fulfilled, and let  $\lest:=\max\{\kappa_{1} \hat u,\frac{\kappa_{2}}{\consttrextwo}\|X^\top \varepsilon\|_\infty\}$. Then, for any  $\hat u\leq \|X^\top Y\|_\infty/\kappa_1$, the prediction loss of the TREX satisfies
\begin{equation*}
\|X\trex-X\betatrue\|^{2}_{2}\le  \left(\frac{1}{\kappa_{1}}+\frac{2}{\kappa_{2}}\right) \|X\lassoat-X\bt\|^{2}_{2} +\left(2+\frac{2}{\kappa_{1}}+\frac{4}{\kappa_{2}}\right)\|X^\top\varepsilon\|_\infty\|\lassoat-\bt\|_1.
\end{equation*}
\end{theorem} 
Theorem~\ref{trexfast} follows from Theorem~\ref{res:fast4} by setting $\kappa_{1}=2$, and $\kappa_{2}=8$.

\begin{proof}[of Theorem~\ref{res:fast4}]
Assume first  $\lasso(\lest)=0$.
Then, since $\hat u\leq \|X^\top Y\|_\infty/\kappa_1$, the definition of the TREX implies
\begin{equation*}
\|Y-X\trex\|^{2}_{2}+\consttrextwo\hat u\|\trex\|_{1}\le \|Y\|_2^2/\kappa_1 =\|Y-X\lassoat\|^{2}_{2}/\kappa_1\,.
\end{equation*}
Assume now $\lasso(\lest)\neq 0$. In view of the KKT conditions for LASSO, $\lassoat$ fulfills 
\begin{equation*}
  \|X^\top (Y-X\lassoat)\|_\infty=\lest=\max\{\kappa_{1} \hat u,\frac{\kappa_{2}}{\consttrextwo}\|X^\top \varepsilon\|_\infty\}\geq \kappa_{1}\hat u.
\end{equation*}
The definition of  the TREX therefore yields
\begin{equation*}
\|Y-X\trex\|^{2}_{2}+\consttrextwo\hat u\|\trex\|_{1}\le \frac{\|Y-X\lassoat\|^{2}_{2}}{{\kappa_{1}}}+\consttrextwo\hat u\|\lassoat\|_1.
\end{equation*}
We now observe that since $\kappa_1> 1$ by assumption, we have $\lest\geq \hat u$, and one can verify easily that this implies $\|\lasso(\hat u)\|_{1}\geq \|\lassoat\|_1$. At the same time, Lemma~\ref{res:l1trexlasso} ensures $\|\trex\|_1\geq \|\lasso^{\hat u}\|_1$.  Thus, $\consttrextwo\hat u\|\trex\|_{1}\ge\consttrextwo\hat u\|\lassoat\|_1$ and, therefore, we find again
\begin{equation*}
\|Y-X\trex\|^{2}_{2}\le \frac{\|Y-X\lassoat\|^{2}_{2}}{\kappa_{1}}.
\end{equation*}
Invoking the model, this results in
\begin{align*}
&\kappa_1\|X\trex-X\betatrue\|^{2}_{2}\\
\le &\left(1-\kappa_{1}\right)\|\varepsilon\|^{2}_{2}+{2\varepsilon^\top (X\bt - X\lassoat)}+2\varepsilon^\top (X\trex - X\bt)+{\|X\lassoat-X\bt\|^{2}_{2}}.
\end{align*}
We can now use H\"older's inequality and the triangle inequality to deduce
\begin{align*}
&\|X\trex-X\bt\|^{2}_{2}\\\le& \left(\frac{1-\kappa_{1}}{\kappa_{1}}\right)\|\varepsilon\|_2^2+\frac{2\|X^\top\varepsilon\|_\infty\|\bt-\lassoat\|_1}{\kappa_{1}}\\
&+ 2\|X^\top\varepsilon\|_\infty\|\trex\|_1+2\|X^\top\varepsilon\|_\infty\|\bt\|_1+ \frac{\|X\lassoat-X\bt\|^{2}_{2}}{\kappa_1} \\
\leq & \left(\frac{1-\kappa_{1}}{\kappa_{1}}\right)\|\varepsilon\|_2^2+\frac{2\|X^\top\varepsilon\|_\infty\|\bt-\lassoat\|_1}{\kappa_{1}}\\
&+ 2\|X^\top\varepsilon\|_\infty\left(\frac{\|Y-X\trex \|_2^2}{\consttrextwo\hat u}+ \|\trex\|_1\right)+2\|X^\top\varepsilon\|_\infty\|\bt\|_1+ \frac{\|X\lassoat-X\bt\|^{2}_{2}}{\kappa_1}.
\end{align*}
Next, we observe that by the definition of our estimator~$\trex$ and of~$\lest$,
\begin{align*}
& \frac{\|Y-X\trex \|_2^2}{\consttrextwo\hat u}+ \|\trex\|_1\leq\frac{\|Y-X\lassoat \|_2^2}{\consttrextwo\lest}+ \|\lassoat\|_1\leq\frac{\|Y-X\lassoat \|_2^2}{\kappa_2\|X^\top\varepsilon\|_\infty}+ \|\lassoat\|_1.
\end{align*}
Combining these two displays and using the model assumption then gives
\begin{align*}
&\|X\trex-X\bt\|^{2}_{2}\\
\leq & \left(\frac{1-\kappa_{1}}{\kappa_{1}}\right)\|\varepsilon\|_2^2+\frac{2\|X^\top\varepsilon\|_\infty\|\bt-\lassoat\|_1}{\kappa_{1}}\\
&+ 2\|X^\top\varepsilon\|_\infty\left(\frac{\|Y-X\lassoat \|_2^2}{\kappa_2\|X^\top\varepsilon\|_\infty}+ \|\lassoat\|_1\right)+2\|X^\top\varepsilon\|_\infty\|\bt\|_1+ \frac{\|X\lassoat-X\bt\|^{2}_{2}}{\kappa_1}\\
=& \left(\frac{1-\kappa_{1}}{\kappa_{1}}+\frac{2}{\kappa_2}\right)\|\varepsilon\|_2^2+\frac{2\|X^\top\varepsilon\|_\infty\|\bt-\lassoat\|_1}{\kappa_{1}}\\
&+ 2\|X^\top\varepsilon\|_\infty\left(\frac{2\varepsilon^\top X(\bt-\lassoat)}{\kappa_2\|X^\top\varepsilon\|_\infty}+ \|\lassoat\|_1\right)+2\|X^\top\varepsilon\|_\infty\|\bt\|_1\\
&+\left(\frac{1}{\kappa_1}+\frac{2}{\kappa_2}\right){\|X\lassoat-X\bt\|^{2}_{2}}.
\end{align*}
We can now use H\"older's inequality and rearrange the terms to get
\begin{align*}
&\|X\trex-X\bt\|^{2}_{2}\\
\leq& \left(\frac{1-\kappa_{1}}{\kappa_{1}}+\frac{2}{\kappa_2}\right)\|\varepsilon\|_2^2+\left(\frac{1}{\kappa_1}+\frac{2}{\kappa_2}\right)2\|X^\top\varepsilon\|_\infty\|\bt-\lassoat\|_1\\
&+ 2\|X^\top\varepsilon\|_\infty\|\lassoat\|_1+2\|X^\top\varepsilon\|_\infty\|\bt\|_1+\left(\frac{1}{\kappa_1}+\frac{2}{\kappa_2}\right){\|X\lassoat-X\bt\|^{2}_{2}}.
\end{align*}
The last step is to use Assumption~\ref{signalstrengthup_kappa}, which ensures
\begin{equation*}
\|\bt\|_1\le\left(\frac14-\frac1{4\kappa_{1}}-\frac1{2\kappa_{2}}\right)\frac{\|\varepsilon\|_2^2}{\|X^{\top}\varepsilon\|_{\infty}}
\end{equation*}
and, therefore,
\begin{equation*}
4\|X^\top\varepsilon\|_\infty\|\bt\|_1\le\left(\frac{\kappa_{1}-1}{\kappa_{1}}-\frac2{\kappa_{2}}\right)\|\varepsilon\|_2^2.
\end{equation*}
The above display therefore yields
\begin{align*}
&\|X\trex-X\bt\|^{2}_{2}\leq -4\|X^\top\varepsilon\|_\infty\|\bt\|_1 +\left(\frac{1}{\kappa_1}+\frac{2}{\kappa_2}\right)2\|X^\top\varepsilon\|_\infty\|\bt-\lassoat\|_1\\
&~~~~+ 2\|X^\top\varepsilon\|_\infty\|\lassoat\|_1+2\|X^\top\varepsilon\|_\infty\|\bt\|_1+\left(\frac{1}{\kappa_1}+\frac{2}{\kappa_2}\right){\|X\lassoat-X\bt\|^{2}_{2}}.
\end{align*}
Using the triangle inequality, we finally obtain
\begin{align*}
&\|X\trex-X\bt\|^{2}_{2}\leq \left(2+\frac{2}{\kappa_1}+\frac{4}{\kappa_2}\right)\|X^\top\varepsilon\|_\infty\|\lassoat-\bt\|_1+\left(\frac{1}{\kappa_1}+\frac{2}{\kappa_2}\right){\|X\lassoat-X\bt\|^{2}_{2}}
\end{align*}
as desired.
\end{proof}

\begin{corollary}\label{cor:fastTgen}  Let Assumption~\ref{signalstrengthup_kappa} be fulfilled, and let  $\lest:=\max\{\kappa_{1} \hat u,\frac{\kappa_{2}}{\consttrextwo}\|X^\top \varepsilon\|_\infty\}$. Furthermore, let $\kappa_1,\kappa_2>0$ be such that
\newjl{
$$
\frac1{\kappa_2}+\frac{\kappa_1}{\kappa_2 + 2\kappa_1} \le \frac{1}{c}.
$$
 }
Then for any  $\hat u\leq \|X^\top Y\|_\infty/\kappa_1$, the prediction loss of TREX satisfies
\begin{equation*}
\|X\trex-X\bt\|^{2}_{2}\leq \left(\frac{1}{\kappa_1}+\frac{2}{\kappa_2}\right)\frac{16s\lambdatrex^2}{\nu^2 n},
\end{equation*}
where $\nu$ is the compatibility constant defined in~\eqref{compcond}.
\end{corollary} 
Corollary~\ref{cor:fastT} follows from Corollary~\ref{cor:fastTgen} by setting $\kappa_{1}=2$, and $\kappa_{2}=8$, which satisfy the requirement for any $c\in (0,2)$.
\begin{proof}[of Corollary~\ref{cor:fastTgen}]
Using Theorem~\ref{res:fast4} and the definition of $\lambdatrex$, the TREX prediction loss satisfies
\begin{align*}
\|X\trex-X\bt\|^{2}_{2}&\leq \left(\frac{1}{\kappa_1}+\frac{2}{\kappa_2}\right){\|X\lassoat-X\bt\|^{2}_{2}}+ \left(2+\frac{2}{\kappa_1}+\frac{4}{\kappa_2}\right)\|X^\top\varepsilon\|_\infty\|\lassoat-\bt\|_1\\
&\leq \left(\frac{1}{\kappa_1}+\frac{2}{\kappa_2}\right)\left[\|X\lassoat-X\bt\|^{2}_{2} + \left(2+\frac{2\kappa_1\kappa_2}{\kappa_2 + 2\kappa_1}\right)\frac{c}{\newjl{\kappa_2 }}\lambdatrex \|\lassoat-\bt\|_1\right].
\end{align*}
On the other hand, the LASSO estimator $\lassoat$ satisfies \cite[Theorem~6.1]{Buhlmann11}
$$
\|X\lassoat-X\bt\|^{2}_{2} +2\lest\|\lassoat-\bt\|_1 \le \frac{16s\lest^2}{\nu^2n}.
$$
Since by assumption
$$
\left(2 + \frac{2\kappa_1\kappa_2}{\kappa_2 + 2\kappa_1}\right)\frac{c}{\newjl{\kappa_2} }\le 2,
$$
the LASSO bound implies
\begin{align*}
\|X\trex-X\bt\|^{2}_{2}&\leq \left(\frac{1}{\kappa_1}+\frac{2}{\kappa_2}\right)\frac{16s\lambdatrex^2}{\nu^2 n}
\end{align*}
as desired.
\end{proof}

\section{Proof of a Generalization of Theorem~\ref{slowratetrex}}\label{s:slowrate}
We consider a generalization of the TREX according to
\begin{align*}
  \trex\in\argmin\left\{\frac{\|Y-X\beta\|^2_2}{\consttrextwo\Omega^*(X^\top(Y-X\beta))}+\Omega(\beta)\right\},
\end{align*}
where $\Omega$ is a norm on $\rp$, $\Omega^{*}(\eta):=\sup\{\eta^{\top}\beta:\Omega(\beta)\le 1\}$ is the dual of $\Omega$, $0<\consttrextwo<2$, and the minimum is taken over all~$\beta\in\rp$. 
We also set $\hat u:=\Omega^{*}(X^{\top}(Y-X\trex))$ with some abuse of notation. The corresponding generalization of Assumption~\ref{signalstrengthlow} then reads as follows.
 \begin{assumption}\label{signalstrengthlow_c} The regression vector $\bt$ is sufficiently large such that
    \begin{equation*}
  \Omega^{*}(X^{\top}X\bt)\ge \left(1+\frac2{\consttrextwo}\right)\Omega^{*}(X^{\top}\varepsilon).    
    \end{equation*}
 \end{assumption}
We now prove a generalization of Theorem~\ref{slowratetrex}.
\begin{theorem}\label{res:bound}  Let Assumption~\ref{signalstrengthlow_c} be fulfilled. If $\hat u\leq\Omega^{*}(X^{\top}Y)$, then the prediction loss of TREX satisfies
\begin{equation*}
\frac{\|X\trex-X\betatrue\|_{2}^{2}}{n}\le \frac{\left(2\Omega^{*}(X^{\top}\varepsilon)+\max\left\{\hat u,\frac2{\consttrextwo}\Omega^{*}(X^{\top}\varepsilon)\right\}\right)\Omega(\betatrue)}{n}.
\end{equation*}
\end{theorem}
\noindent Theorem~\ref{slowratetrex} follows from Theorem~\ref{res:bound} by setting $\Omega(\cdot):=\|\cdot\|_{1}.$

\begin{proof}[of Theorem~\ref{res:bound}]
The definition of the estimator implies
\begin{equation*}
    \frac{\|Y-X\trex\|^2_2}{\consttrextwo\cdot \hat u}+\Omega(\trex)\leq\frac{\|Y\|^2_2}{\consttrextwo\cdot\Omega^*(X^\top Y)},
\end{equation*}
which yields together with  the model assumptions
\begin{align*}
&\|X\trex-X\betatrue\|^{2}_{2}+\|\varepsilon\|^{2}_{2}+2\varepsilon^{\top}(X\bt-X\trex)+\consttrextwo\hat u\Omega(\trex)\\
\le& \frac{\hat u}{\Omega^{*}(X^{\top}Y)}\left(\|\varepsilon\|^{2}_{2}+\|X\betatrue\|_{2}^{2}+2\varepsilon^{\top}X\betatrue\right).
\end{align*}
Rearranging the terms and H\"{o}lder's inequality in the form of  $2\varepsilon^\top X\trex\leq 2\Omega^*(X^\top\varepsilon)\Omega(\trex)$ and $\|X\betatrue\|^{2}_{2}=\betatrue{}^{\top}X^{\top}X\betatrue\le \Omega^{*}(X^{\top}X\betatrue)\Omega(\betatrue)$ then gives
\begin{equation}\label{eq:interim2}
\begin{split}
\|X\trex-X\betatrue\|^{2}_{2}\le &\left(\frac{\hat u}{\Omega^{*}(X^{\top}Y)}-1\right)\|\varepsilon\|^{2}_{2}+(2\Omega^{*}(X^{\top}\varepsilon)-\consttrextwo\hat u)\Omega(\trex)\\
&+\frac{\hat u\Omega^{*}(X^{\top}X\betatrue)\Omega(\betatrue)}{\Omega^{*}(X^{\top}Y)}+2\left(\frac{\hat u}{\Omega^{*}(X^{\top}Y)}-1\right)2\varepsilon^{\top}X\betatrue.
\end{split}
\end{equation}
{\bf Case 1:} We first consider the case $2\Omega^{*}(X^{\top}\varepsilon)\leq \consttrextwo\hat u$.\\
For this, we first note that $\hat u \le \Omega^{*}(X^{\top}Y)$ by assumption. Using this and $2\Omega^{*}(X^{\top}\varepsilon)\leq \consttrextwo\hat u$ allows us to remove the first two terms on the right hand side of Inequality~\eqref{eq:interim2} so that
\begin{align*}
\|X\trex-X\betatrue\|^{2}_{2}\le \frac{\hat u\Omega^{*}(X^{\top}X\betatrue)\Omega(\betatrue)}{\Omega^{*}(X^{\top}Y)}+\left(\frac{\hat u}{\Omega^{*}(X^{\top}Y)}-1\right)2\varepsilon^{\top}X\betatrue.
\end{align*}
Since $2\varepsilon^\top X\betatrue\leq 2\Omega^{*}(X^{\top}\varepsilon)\Omega(\betatrue)$ due to H\"older's Inequality, and since $\hat u\leq \Omega^*(X^\top Y)$ by definition of our estimator, we therefore obtain from the above display and the model assumptions
\begin{align*}
\|X\trex-X\betatrue\|^{2}_{2}&\le \frac{\hat u\Omega^{*}(X^{\top}X\betatrue)\Omega(\betatrue)}{\Omega^{*}(X^{\top}Y)}+\left(1-\frac{\hat u}{\Omega^{*}(X^{\top}Y)}\right)2\Omega^{*}(X^{\top}\varepsilon)\Omega(\betatrue)\\
&=2\Omega^{*}(X^{\top}\varepsilon)\Omega(\betatrue)+\hat u\left(\frac{\Omega^{*}(X^{\top}X\betatrue)-2\Omega^{*}(X^{\top}\varepsilon)}{\Omega^{*}(X^{\top}Y)}\right)\Omega(\betatrue)\\
&=2\Omega^{*}(X^{\top}\varepsilon)\Omega(\betatrue)+\hat u\left(\frac{\Omega^{*}(X^{\top}X\betatrue)-2\Omega^{*}(X^{\top}\varepsilon)}{\Omega^{*}(X^{\top}X\betatrue+X^{\top}\varepsilon)}\right)\Omega(\betatrue).
\end{align*}
Next, we note that the triangle inequality gives
$$\Omega^{*}(X^{\top}X\betatrue+X^{\top}\varepsilon)\ge \Omega^{*}(X^{\top}X\betatrue)-\Omega^{*}(X^{\top}\varepsilon).$$
Plugging this in the previous display finally yields
\begin{align*}
\|X\trex-X\betatrue\|^{2}_{2}\le 2\Omega^{*}(X^{\top}\varepsilon)\Omega(\betatrue)+\hat u\Omega(\betatrue),
\end{align*}
which concludes the proof for Case 1.\\
\noindent
{\bf Case 2:} We now consider the case $2\Omega^{*}(X^{\top}\varepsilon)\geq \consttrextwo\hat u$.\\
Similarly as before, we start with the definition of the estimator, which yields in particular
\begin{equation*}
\Omega(\trex)\leq\frac{\|Y\|^2_2}{\consttrextwo\cdot\Omega^*(X^\top Y)}.
\end{equation*}
Invoking the model assumptions and H\"older's inequality then gives
\begin{align*}
\Omega(\trex)&\le\frac{1}{\consttrextwo\Omega^{*}(X^{\top}Y)}\left(\|\varepsilon\|^{2}_{2}+\|X\betatrue\|^{2}_{2}+2\varepsilon^{\top}X\betatrue\right)\\
&\le\frac{1}{\consttrextwo\Omega^{*}(X^{\top}Y)}\left(\|\varepsilon\|^{2}_{2}+\Omega^{*}(X^{\top}X\betatrue)\Omega(\betatrue)+2\varepsilon^\top X\betatrue\right).
\end{align*}
We can now plug this into Inequality~\eqref{eq:interim2} to obtain
\begin{align*}
&\|X\trex-X\betatrue\|^{2}_{2}\le \left(\frac{\hat u}{\Omega^{*}(X^{\top}Y)}-1\right)\|\varepsilon\|^{2}_{2}\\
&~~~~+(2\Omega^*(X^\top\varepsilon)-\consttrextwo\hat u)\frac{1}{\consttrextwo\Omega^{*}(X^{\top}Y)}\left(\|\varepsilon\|^{2}_{2}+\Omega^{*}(X^{\top}X\betatrue)\Omega(\betatrue)+2\varepsilon^{\top}X\betatrue\right)\\
&~~~~+\frac{\hat u\Omega^{*}(X^{\top}X\betatrue)\Omega(\betatrue)}{\Omega^{*}(X^{\top}Y)}+\left(\frac{\hat u}{\Omega^{*}(X^{\top}Y)}-1\right)2\varepsilon^{\top}X\betatrue.
\end{align*}
We can now  rearrange the terms to get
\begin{align*}
\|X\trex-X\betatrue\|^{2}_{2}\le& \left(\frac{2\Omega^{*}(X^{\top}\varepsilon)}{\consttrextwo\Omega^{*}(X^{\top}Y)}-1\right)\|\varepsilon\|^{2}_{2}+\left(\frac{2\Omega^{*}(X^{\top}\varepsilon)}{\consttrextwo\Omega^{*}(X^{\top}Y)}-1\right)2\varepsilon^{\top}X\betatrue\\
&+\frac{2\Omega^{*}(X^{\top}\varepsilon)}{\consttrextwo\Omega^{*}(X^{\top}Y)}\Omega^{*}(X^{\top}X\betatrue)\Omega(\betatrue).
\end{align*}
We now observe that Assumption~\ref{signalstrengthlow_c} implies via the triangle inequality and the model assumptions that
\begin{equation*}
  \Omega^{*}(X^{\top}Y)\ge \frac{2\Omega^{*}(X^{\top}\varepsilon)}{c}.    
    \end{equation*}
Using this, H\"older's inequality, and the triangle inequality, we then find
\begin{align*}
&\|X\trex-X\betatrue\|^{2}_{2}\\
\le& \left(1-\frac{2\Omega^{*}(X^{\top}\varepsilon)}{\consttrextwo\Omega^{*}(X^{\top}Y)}\right)2\Omega^*(X^{\top}\varepsilon)\Omega(\betatrue)\\
&+2\Omega^*(X^\top \varepsilon)\Omega(\bt)/\consttrextwo+\frac{2\Omega^*(X^\top \varepsilon)}{\consttrextwo\Omega^*(X^\top Y)}\Omega^*(X^\top \varepsilon)\Omega(\bt)\\
=& 2\Omega^*(X^{\top}\varepsilon)\Omega(\betatrue)+2\Omega^*(X^\top \varepsilon)\Omega(\bt)/\consttrextwo-\frac{2\Omega^*(X^\top \varepsilon)}{\consttrextwo\Omega^*(X^\top Y)}\Omega^*(X^\top \varepsilon)\Omega(\bt)\\
\leq &(2+2/\consttrextwo)\Omega^*(X^{\top}\varepsilon)\Omega(\betatrue).
\end{align*}
This concludes the proof for Case~2 and, therefore, the proof of Theorem~\ref{res:bound}.
\end{proof}

%
%

\end{document}